\documentclass[12pt]{amsart}
\usepackage{amsmath,amssymb,amsbsy,amsfonts,amsthm,latexsym, mathtools,amsopn,amstext,amsxtra,euscript,amscd,mathrsfs,color,bm, cite,multirow,tabularx}
\usepackage{graphicx} 
\usepackage{todonotes,dirtytalk}   
\usepackage{url}
\usepackage[colorlinks,linkcolor=blue,anchorcolor=blue,citecolor=blue]{hyperref}
\usepackage{color}
\usepackage{comment}
\allowdisplaybreaks

\DeclareMathOperator{\li}{li}

\begin{document}
\newtheorem{problem}{Problem}
\newtheorem{theorem}{Theorem}
\newtheorem{lemma}[theorem]{Lemma}
\newtheorem{crit}[theorem]{Criterion}
\newtheorem{claim}[theorem]{Claim}
\newtheorem{cor}[theorem]{Corollary}
\newtheorem{prop}[theorem]{Proposition}
\newtheorem{definition}{Definition}
\newtheorem{question}[theorem]{Question}
\newtheorem{Note}[theorem]{Notation}

\theoremstyle{remark} 
\newtheorem*{note}{Note}
\newtheorem*{rem}{Remark}

\def\cA{{\mathcal A}}
\def\cB{{\mathcal B}}
\def\cC{{\mathcal C}}
\def\cD{{\mathcal D}}
\def\cE{{\mathcal E}}
\def\cF{{\mathcal F}}
\def\cG{{\mathcal G}}
\def\cH{{\mathcal H}}
\def\cI{{\mathcal I}}
\def\cJ{{\mathcal J}}
\def\cK{{\mathcal K}}
\def\cL{{\mathcal L}}
\def\cM{{\mathcal M}}
\def\cN{{\mathcal N}}
\def\cO{{\mathcal O}}
\def\cP{{\mathcal P}}
\def\cQ{{\mathcal Q}}
\def\cR{{\mathcal R}}
\def\cS{{\mathcal S}}
\def\cT{{\mathcal T}}
\def\cU{{\mathcal U}}
\def\cV{{\mathcal V}}
\def\cW{{\mathcal W}}
\def\cX{{\mathcal X}}
\def\cY{{\mathcal Y}}
\def\cZ{{\mathcal Z}}

\def\A{{\mathbb A}}
\def\B{{\mathbb B}}
\def\C{{\mathbb C}}
\def\D{{\mathbb D}}
\def\E{{\mathbb E}}
\def\F{{\mathbb F}}
\def\G{{\mathbb G}}
\def\I{{\mathbb I}}
\def\J{{\mathbb J}}
\def\K{{\mathbb K}}
\def\L{{\mathbb L}}
\def\M{{\mathbb M}}
\def\N{{\mathbb N}}
\def\O{{\mathbb O}}
\def\P{{\mathbb P}}
\def\Q{{\mathbb Q}}
\def\R{{\mathbb R}}
\def\S{{\mathbb S}}
\def\T{{\mathbb T}}
\def\U{{\mathbb U}}
\def\V{{\mathbb V}}
\def\W{{\mathbb W}}
\def\X{{\mathbb X}}
\def\Y{{\mathbb Y}}
\def\Z{{\mathbb Z}}

\def\ep{{\mathbf{e}}_p}
\def\eq{{\mathbf{e}}_q}
\def\cal#1{\mathcal{#1}}

\def\scr{\scriptstyle}
\def\\{\cr}
\def\({\left(}
\def\){\right)}
\def\[{\left[}
\def\]{\right]}
\def\<{\langle}
\def\>{\rangle}
\def\fl#1{\left\lfloor#1\right\rfloor}
\def\rf#1{\left\lceil#1\right\rceil}
\def\le{\leqslant}
\def\ge{\geqslant}
\def\eps{\varepsilon}
\def\mand{\qquad\mbox{and}\qquad}

\def\sssum{\mathop{\sum\ \sum\ \sum}}
\def\ssum{\mathop{\sum\, \sum}}
\def\ssumw{\mathop{\sum\qquad \sum}}

\def\vec#1{\mathbf{#1}}
\def\inv#1{\overline{#1}}
\def\num#1{\mathrm{num}(#1)}
\def\dist{\mathrm{dist}}

\def\fA{{\mathfrak A}}
\def\fB{{\mathfrak B}}
\def\fC{{\mathfrak C}}
\def\fU{{\mathfrak U}}
\def\fV{{\mathfrak V}}

\newcommand{\bflambda}{{\boldsymbol{\lambda}}}
\newcommand{\bfxi}{{\boldsymbol{\xi}}}
\newcommand{\bfrho}{{\boldsymbol{\rho}}}
\newcommand{\bfnu}{{\boldsymbol{\nu}}}

\def\GL{\mathrm{GL}}
\def\SL{\mathrm{SL}}

\def\Hba{\overline{\cH}_{a,m}}
\def\Hta{\widetilde{\cH}_{a,m}}
\def\Hb1{\overline{\cH}_{m}}
\def\Ht1{\widetilde{\cH}_{m}}

\def\flp#1{{\left\langle#1\right\rangle}_p}
\def\flm#1{{\left\langle#1\right\rangle}_m}

\def\Zm{\Z/m\Z}

\def\Err{{\mathbf{E}}}
\def\O{\mathcal{O}}

\def\cc#1{\textcolor{red}{#1}}
\newcommand{\commT}[2][]{\todo[#1,color=green!60]{Tim: #2}}
\newcommand{\commB}[2][]{\todo[#1,color=red!60]{Bryce: #2}}

\DeclarePairedDelimiter\ceil{\lceil}{\rceil}
\DeclarePairedDelimiter\floor{\lfloor}{\rfloor}

\newcommand{\comm}[1]{\marginpar{%
\vskip-\baselineskip 
\raggedright\footnotesize
\itshape\hrule\smallskip#1\par\smallskip\hrule}}
\newcolumntype{L}{>{\raggedright\arraybackslash}X}

\def\xxx{\vskip5pt\hrule\vskip5pt}

\def\dmod#1{\,\left(\textnormal{mod }{#1}\right)}


\title{Improved explicit Burgess-strength bounds for composite moduli}

\author[J. Punch]{James Punch}\thanks{}
\address{School of Science, The University of New South Wales Canberra, Australia}
\email{j.punch@unsw.edu.au}

\pagenumbering{arabic}

\begin{abstract}
Consider the character sum $S(M,N)=\sum_{n=M+1}^{M+N}\chi(n)$ for a non-principal Dirichlet character $\chi(n)$ to a (possibly composite) modulus $q$. Burgess \cite{Burgess1963} proved that $S(M,N)\ll N^{1/2}q^{3/16+\varepsilon}$, for any $\varepsilon>0$. This was made explicit by Jain-Sharma et al.\ in \cite{JSKL2021}, but only for moduli $q\geq \exp(\exp(9.594))\approx 8.03\times 10^{6373}$. We prove a stronger result that holds when $q\geq 1.627\times 10^{19}$.

\end{abstract}

\maketitle
\let\thefootnote\relax
\footnote{\textit{Affiliation}: School of Science, University of New South Wales, Canberra.}
\footnote{\textit{Corresponding author}:}
\footnote{\textit{MSC2020}:}
\footnote{\textit{Keywords}:}

\section{Introduction}
It is of interest in many number-theoretic problems to study sums of Dirichlet characters over a set of consecutive integers. For a non-principal character $\chi$ whose modulus is some (possibly composite) integer $q$, define $$S(M,N):=\sum_{n=M+1}^{M+N}\chi(n).$$ A trivial upper bound for this sum's modulus is $N$, which tends to be best for small $N$. For large $N$, we instead employ the P\'olya--Vinogradov inequality, which states that $S(M,N)\ll \sqrt{q}\log{q}$. An improvement for moderately large $N$ was given by Burgess in \cite{Burgess1963}. Let $r\geq 2$ be an integer; he proved that if $q$ is cubefree or $r=2$, then \begin{align}\label{eq:BurgessComposite}S(M,N)\ll N^{1-\frac{1}{r}}q^{\frac{r+1}{4r^2}+\varepsilon}\end{align} for any $\varepsilon>0$. Further, if $q$ is prime, then $$S(M,N)\ll N^{1-\frac{1}{r}}q^{\frac{r+1}{4r^2}}\log{q}.$$ The power of $\log{q}$ has since been improved; to the author's knowledge, the best published result is that of de la Bret\`eche et al. \cite{delaBreteche2021}, who gave $(\log{q})^{\eta/2r}$ for $\eta\approx 0.16656$.

Since these bounds tend to be useful in applications (e.g.\ bounding least primitive roots, studying Dirichlet $L$-functions, computing class groups and ranks of elliptic curves etc.), we would like to have explicit versions of these results. For a quadratic character to a prime modulus $q$, Booker \cite{Booker2006} gave upper bounds of the form $\alpha(r)N^{1-\frac{1}{r}}q^{\frac{r+1}{4r^2}}(\beta(r)+\log{q})^{\frac{1}{2r}}$ for integers $2\leq r \leq 15$ and primes $q\geq q_0$, with explicit values for $\alpha(r)$, $\beta(r)$ and $q_0$. Francis's upper bounds in \cite{Francis2021} (which improve on work of McGown \cite{McGown2012} and Trevi\~no \cite{Trevino2015}) are similar but valid for all non-principal characters: $\alpha(r)N^{1-\frac{1}{r}}q^{\frac{r+1}{4r^2}}(\log{q})^\frac{1}{2r}$ for $2\leq r \leq 10$, primes $q\geq q_0(r)$, with $\alpha(r)$ and $q_0(r)$ explicit.

Less has been achieved in the case where $q$ is allowed to be composite. When $r=2$, Jain-Sharma et al.\ provided the following bound, which is an explicit version of \eqref{eq:BurgessComposite}. Let $\omega(q)$ be the number of distinct primes dividing $q$, $d(q)$ the number of divisors of $q$, and $\varphi(q)$ the number of integers coprime to $q$ in $[1,q-1]$.
\begin{theorem} \cite[Theorem 1.1]{JSKL2021}\label{prop:JainSharma}
    Let $\chi$ be a primitive character with modulus $q\geq \exp(\exp(9.594))$. Define $$\lambda_0=9.07 \cdot (8^{\omega(q)}d(q)^3)^{1/4}(\log{q})^{1/4}\left(\frac{q}{\phi(q)}\right)^{1/2}.$$ If $N\leq q^{5/8}$, then $$\lvert S(M,N)\rvert \leq N^{1/2} q^{3/16}\lambda_0.$$
\end{theorem}
This paper's objective is to establish the following improvement to Theorem \ref{prop:JainSharma}.
\begin{theorem} \label{prop:mainTheorem}
    Let $\chi$ be a primitive character with modulus $q\geq 1.627\times 10^{19}$. If $p_1^{r_1}\ldots p_s^{r_s}$ is the prime factorisation of $q$, then let $T(q)=\binom{r_1+3}{3}\ldots \binom{r_s+3}{3}$. Define $$\lambda_1=4.0733\cdot (8^{\omega(q)}T(q))^{1/4}(2.578+\log{q})^{1/4}.$$ If $N\leq q^{5/8}$, then $$\lvert S(M,N)\rvert \leq N^{1/2} q^{3/16} \lambda_1.$$
\end{theorem}
One application arises in \cite{Elkies2025}: Elkies uses Burgess-strength bounds for \emph{prime} moduli to verify that the rank of a particular elliptic curve is \emph{exactly} 16; prior to this, it was only known that its rank was \emph{at least} 16. Elkies then gives a curve whose rank is at least 17, but observes that proving its rank equals 17 (which is known under GRH) would require a Burgess-strength result for \emph{composite} moduli. Theorem \ref{prop:JainSharma} is insufficient because of the very high lower bound on $q$, but Theorem \ref{prop:mainTheorem} would suffice.

We note a recent preprint by Hasanalizade et al. \cite{HLMMT2026}, who improved Theorem \ref{prop:JainSharma} by extending the result to $r\geq 3$ for cubefree $q$, removing the restriction on $N$, decreasing the constant $9.07$ slightly, and permitting $q\geq 10^{1143}$ for $r$ not too large.
\section{Our strategy}

The Burgess estimate is proved by considering shifts of $S(M,N)$ by some integer $ab$; typically, $b$ is allowed to take any integer value in a particular range (e.g.\ $[1,B]$ and $[-B,B]$, for some $B\geq 1$). The set over which $a$ can vary is usually restricted in some way, such as the primes or integers coprime to $q$ in $[1,A]$ or $[-A,A]$. Such a choice is made in \cite{JSKL2021}; they allow $a$ to vary over $$A^q:=\{n\in \Z : 1\leq n \leq A, \gcd(n,q)=1\}.$$ This set's cardinality, which they denote $A_q$, is roughly $A\cdot \frac{\varphi(q)}{q}$, but the error may be as large as $2^{\omega{(q)}-1}$. This bound is sharp in a certain sense (see \cite{Vijayaraghavan1951}). Later, they prove and use an upper bound $\frac{A}{A_q}\leq 2\frac{q}{\varphi(q)}$, but the slow decay of $2^{\omega(q)-1}$ compared to $q$ means that this bound is only valid when $q\geq \exp(\exp(9.594))$. The lower bound on $q$ in Theorem \ref{prop:JainSharma} originates from this restriction, but it need not be so high; in Section \ref{sec:cheap}, we show that Theorem \ref{prop:JainSharma} still holds for much smaller values of $q$.

In Section \ref{sec:proper}, we instead allow $a$ to vary over the primes in $[1,A]$. The size of this set, $\pi(A)$, is well-studied and has comparatively better error terms available, so we will be able to achieve a significant improvement to the lower bound of $q$.

\section{A cheap improvement upon Jain-Sharma et al.} \label{sec:cheap}
We note that Theorem \ref{prop:JainSharma} actually holds for $q\geq 1.95\times 10^{446}$, and this can be proved without much trouble. The bottleneck stems from \cite[Proposition 2.3]{JSKL2021}:
\begin{prop}\label{prop:JainSharma2.3}
    Suppose that $q$ is a positive integer with $q\geq \exp(\exp(9.594))$ and that $N$ is a positive integer in the range $[q^{3/8},q^{5/8}]$. Let $A=\floor{\frac{1}{10}Nq^{-1/4}}$. Then $\frac{A_q}{A}\geq \frac{1}{2}\cdot \frac{\varphi(q)}{q}$.
\end{prop}
To improve this, we use results of Robin and of Rosser and Schoenfeld. Both can be improved by assuming a larger lower bound on $n$, but our goal here is merely to show that Theorem \ref{prop:JainSharma} holds for much smaller $q$ than claimed in \cite{JSKL2021}.
\begin{lemma}\cite[Theorem 13]{Robin1983}\label{prop:Robin}
Let $n$ be a positive integer such that $n\geq 26$. Then $$\omega(n)\leq \frac{\log{n}}{-1.1714+\log{\log{n}}}.$$
\end{lemma}
\begin{lemma}\cite[(3.42)]{RS1962}\label{prop:RS-NOverPhiN}
Let $n$ be a positive integer such that $n\geq 3$. Then $$\frac{n}{\varphi(n)} \leq e^{\gamma}\log{\log{n}}+\frac{2.50637}{\log{\log{n}}},$$ where $\gamma$ is Euler's constant.    
\end{lemma}
We now prove that Proposition \ref{prop:JainSharma2.3} still holds if we relax the lower bound on $q$ to $1.95\times 10^{446}$. As mentioned above, we have $$\frac{A_q}{A}>\frac{\varphi(q)}{q}-\frac{2^{\omega(q)-1}}{A},$$ so we will be done if $$\frac{\varphi(q)}{q}-\frac{2^{\omega(q)-1}}{A}\geq \frac{1}{2}\frac{\varphi(q)}{q}\impliedby A\geq \frac{2^{\omega(q)}q}{\varphi(q)}\impliedby q^{1/8}\geq 10\left(\frac{2^{\omega(q)}q}{\varphi(q)}+1\right).$$
Write $q=\exp(\exp(K))$ and use Propositions \ref{prop:Robin} and \ref{prop:RS-NOverPhiN}; it is enough to have \begin{align}\label{eq:2.3WinCondition}e^{\frac{1}{8}e^K}\geq 10\left(e^{e^K\cdot \frac{0.6932}{K-1.1714}}\left(e^{\gamma}K+\frac{2.50637}{K}\right)+1\right).\end{align} Since $\exp(\exp(6.935))<1.95\times 10^{446}$, we will be done if we can show that \eqref{eq:2.3WinCondition} holds when $K\geq6.935$. We assume this and find that the right-hand side of \eqref{eq:2.3WinCondition} does not exceed \begin{align*}10\left(e^{0.1202721e^K}(1.78109K+0.36141)+1\right)&\leq 18.33204Ke^{0.1202721e^K}+10\\
&=e^{e^K\left[0.1202721+\frac{\log{18.33204}}{e^K}+\frac{\log{K}}{e^K}\right]}+10\\
&\leq e^{0.124988e^K}+10,
\end{align*} which is easily proven to be less than or equal to $e^{0.125e^K}$, as needed. It is not difficult to see that the proof of \cite[Theorem 1.1]{JSKL2021} still works out if we assume the relaxed lower bound and substitute our improved version of Proposition \ref{prop:JainSharma2.3}.
\section{A more considered improvement} \label{sec:proper}
In this section, we prove Theorem \ref{prop:mainTheorem}. First, we present a result due to Burgess.
\begin{lemma} \cite[Lemma 7]{Burgess1962} \label{prop:Weil}
    Let $\chi$ be a primitive character with modulus $q$, and let $$f(x)=(x-b_1)(x-b_2)(x-b_3)^{\varphi(q)-1}(x-b_4)^{\varphi(q)-1},$$ with at least three of the $b_i$ being distinct. For any $i\in\{1,2,3,4\}$, let $$A_i=A_i(b_1,b_2,b_3,b_4):=\prod_{j\neq i}(b_i-b_j).$$ Then for some nonzero $A_i$, we have $$\left\lvert\sum_{x=1}^{q}\chi(f(x))\right\rvert\leq (q,A_i)\cdot 8^{\omega(q)}\sqrt{q}.$$
\end{lemma}
The next result is elementary; it will be used in the proof of the following lemma.
\begin{lemma}\label{prop:counting}
    Let $S$ be a set of $B$ consecutive integers with $0\in S$. For any integer $n$, there are at most $\frac{B-1}{n}$ nonzero multiples of $n$ in $S$.
\end{lemma}
\begin{proof}
    Write $N$ and $P$, respectively, for the number of negative and positive integers in $S$, so that $N+P+1=B$. There are $\floor{P/n}$ multiples of $n$ among the positive integers and $\floor{N/n}$ multiples of $n$ among the negative integers. The total is $\floor{P/n}+\floor{N/n}\leq \frac{P+N}{n}=\frac{B-1}{n}.$
\end{proof}
Next, we improve \cite[Lemma 3.2]{JSKL2021}. We will use the notation $d_3(n)$ to denote the number of ways to write $n$ as the product of three integers, where order matters.
\begin{lemma} \label{prop:WeilCounting}
    Let $\chi$ be any primitive character modulo $q$ and let $B$ be a positive integer. If $p_1^{r_1}\ldots p_s^{r_s}$ is the prime factorisation of $q$, then let $T(q)=\binom{r_1+3}{3}\ldots \binom{r_s+3}{3}$. Then
    $$\sum_{x=1}^{q}\left\lvert \sum_{b=1}^{B}\chi(x+b)\right\rvert^4\leq (7B^2-6B)q+4\cdot  8^{\omega(q)}B(B-1)^3 T(q)\sqrt{q}.$$
\end{lemma}
\begin{rem}
    In \cite[Lemma 3.2]{JSKL2021}, the $T(q)$ is replaced by $d(q)^3$ and $B(B-1)^3$ is replaced by $B^4$. We note that $d(q)^3=(r_1+1)^3\ldots (r_s+1)^3$, and that $\binom{r+3}{3}< (r+1)^3$ for $r\geq 1$, so our result is an improvement for every $q$.
\end{rem}
\begin{proof}
    The triangle inequality yields $$\sum_{x=1}^{q}\left\lvert \sum_{b=1}^{B}\chi(x+b)\right\rvert^4\leq \sum_{1\leq b_1,\ldots,b_4\leq B}\left\lvert \sum_{x=1}^{q}\chi(x+b_1)\chi(x+b_2)\bar{\chi}(x+b_3)\bar{\chi}(x+b_4)\right\rvert.$$ Suppose that $\#\{b_1,b_2,b_3,b_4\}\leq 2$; this occurs in at exactly $B+\binom{B}{2}(2^4-2)=7B^2-6B$ ways. Bounding the sum over $x$ trivially gives a contribution of $(7B^2-6B)q$. Supposing now that $\#\{b_1,b_2,b_3,b_4\}\geq 3$, we apply Lemma \ref{prop:Weil} with $f(x)=(x+b_1)(x+b_2)(x+b_3)^{\varphi(q)-1}(x+b_4)^{\varphi(q)-1}$: \begin{align*}\sum_{\substack{1\leq b_1,b_2,b_3,b_4 \leq B\\ \#\{b_1,b_2,b_3,b_4\}\geq 3}}\left\lvert \sum_{x=1}^{q}\chi(f(x))\right\rvert &\leq 8^{\omega(q)}\sqrt{q}\sum_{1\leq b_1,\ldots,b_4\leq B}\sum_{\substack{1\leq i\leq 4\\ A_i\neq 0}} (q,A_i)\\
    &\leq 4\cdot 8^{\omega(q)}\sqrt{q}\sum_{\substack{1\leq b_1,\ldots,b_4\leq B\\A_1\neq 0}}(q,A_1)\\
    &\leq 4\cdot 8^{\omega(q)} \sqrt{q} \sum_{\substack{1\leq b_1,\ldots, b_4 \leq B\\A_1\neq 0}}\sum_{\substack{k\mid q\\k\mid A_1}}k\\
    &=4\cdot 8^{\omega(q)}\sqrt{q} \sum_{k\mid q}k\sum_{\substack{1\leq b_1,\ldots,b_4\leq B\\A_1\neq 0 \\ A_1\equiv 0\pmod{k}}}1.
    \end{align*}
    Now, we estimate the inner sum. We must count the number of ways that $A_1=(b_1-b_2)(b_1-b_3)(b_1-b_4)$ can be nonzero and a multiple of $k$. Choose $b_1$; this can be done in $B$ ways. Write $k=k_2 k_3 k_4$ (this can be done in $d_3(k)$ ways); we would like to arrange for $k_2\mid b_1-b_2$, $k_3\mid b_1-b_3$ and $k_4\mid b_1-b_4$. For this, we need $b_1-b_2$ to be a nonzero multiple of $k_2$, and we know it is contained in $\{b_1-B,b_1-B+1,\ldots,b_1-1\}$. This is a set of $B$ consecutive integers containing $0$, so by Lemma \ref{prop:counting} it contains at most $\frac{B-1}{k_2}$ nonzero multiples of $k_2$. Choosing $k_3$ and $k_4$ can similarly be done in at most $\frac{B-1}{k_3}$ and $\frac{B-1}{k_4}$ ways respectively. The inner sum is therefore at most $$B\cdot d_3(k)\cdot \frac{(B-1)^3}{k_2 k_3 k_4}=B(B-1)^3\frac{d_3(k)}{k}.$$ We get $$\sum_{\substack{1\leq b_1,b_2,b_3,b_4 \leq B\\ \#\{b_1,b_2,b_3,b_4\}\geq 3}}\left\lvert \sum_{x=1}^{q}\chi(f(x))\right\rvert \leq 4\cdot 8^{\omega(q)} B(B-1)^3 \sqrt{q} \sum_{k\mid q}  d_3(k);$$ combining this with $(7B^2-6B)q$ yields the result if we can show that $T(q)=\sum_{k\mid q} d_3(k)$. It is easy to see that $d_3(k)$ is multiplicative with respect to $k$, and that $d_3(p^m)=\binom{m+2}{2}$ for a prime power $p^m$. It follows that $\sum_{k\mid q} d_3(k)$ is multiplicative with respect to $q$, and that $$\sum_{k\mid p^m}d_3(k)=\sum_{0\leq j \leq m}d_3(p^j)=\binom{2}{2}+\binom{3}{2}+\ldots+\binom{m+2}{2},$$ which can be seen (via Pascal's triangle, for example) to equal $\binom{m+3}{3}$.
    \end{proof}
We will encounter the function $\nu(x)$, which counts the number of ways to write $n\equiv ax\pmod{q}$ using a prime $a\leq A$ and an integer $n\in [M+1,M+N]$. We now give an upper bound for its second moment, $\sum_x \nu(x)^2$. Write $\li(x)$ for the logarithmic integral function, defined by the Cauchy principal value of the integral $\int_0^{x}\frac{1}{\log{t}}\,dt$.
\begin{lemma} \label{prop:secondMoment}
    Let $A$, $N$ and $q$ be positive integers with $AN<q/2$ and $A\geq 3$. Then $$\sum_{x=1}^{q} \nu(x)^2\leq \pi(A)^2+2N\li{\pi(A)}+N\pi(A)+U(q,A),$$ where $U(q,A):=\sum_{\substack{p\mid q\\ p\leq A}}(p-1)$ and $\li$ is the logarithmic integral function.
\end{lemma}
\begin{proof} In what follows, unless otherwise specified, all modular equivalences will be modulo $q$.
We have $$\sum_{x=1}^{q}\nu(x)^2=\sum_{x=1}^{q}\left[\sum_{\substack{a\leq A\\a\text{ prime}}}\sum_{\substack{M+1\leq n \leq M+N\\n\equiv ax}}1\right]^2
=\sum_{a_1,a_2,n_1,n_2}\sum_{\substack{1\leq x\leq q\\ n_1\equiv a_1x\\ n_2\equiv a_2x}}1.$$ 
Fix $a_1$ and $a_2$, and write $M+n_1'=n_1$ and $M+n_2'=n_2$ for $n_1',n_2'\in[1,N]$. The condition in the inner sum is then equivalent to \begin{align}\label{eq:secondMomentCongruence}M+n_1'\equiv a_1 x,M+n_2'\equiv a_2 x \pmod{q},\end{align} which implies that $(a_1-a_2)M+(n_2' a_1-n_1' a_2)\equiv 0.$ We select $k\equiv (a_1-a_2)M$ such that $\lvert k \rvert \leq q/2$. Since $AN<q/2$, we must have $\lvert n_1'a_2-n_2'a_1 \rvert <q/2$; from which it follows that $n_1'a_2-n_2'a_1=k$. If this is solved by $n_1'=Y$ and $n_2'=Z$, then the general solution is of the form $n_1'=Y+\frac{a_1}{(a_1,a_2)}h$, $n_2'=Z+\frac{a_2}{(a_1,a_2)}h$, where $h\in \Z$. Next, we count the number of ways to choose $n_1',n_2'$, then the number of ways to choose $x$ after this choice. We consider three different cases:
\begin{itemize}
    \item \textbf{At least one of $a_1$ and $a_2$ does not divide $q$.} Without loss of generality, suppose that $a_1$ does not divide $q$. Because $n_1'$ and $n_2'$ must be in $[1,N]$, we can choose $h$ above in at most $$\min\left(1+\frac{N}{a_1/(a_1,a_2)},1+\frac{N}{a_2/(a_1,a_2)}\right)=1+N\frac{(a_1,a_2)}{\max(a_1,a_2)}$$ ways. Having chosen $(n_1',n_2')$, we must select $x\equiv a_1^{-1}(M+n_1')$ (i.e.\ there is only one option).
    \item \textbf{Both $a_1$ and $a_2$ divide $q$, but $a_1\neq a_2$.} As before, there are at most $1+N\frac{(a_1,a_2)}{\max(a_1,a_2)}$ ways to select $h$. We claim that there is at most one $x$ satisfying \eqref{eq:secondMomentCongruence}: suppose that both $x_1$ and $x_2$ satisfy \eqref{eq:secondMomentCongruence}, so that $a_1x_1\equiv a_1x_2$ and $a_2x_1\equiv a_2x_2$. One may write $q=q_1q_2$ such that $(q_1,q_2)=1$ and $a_1\nmid q_2$, $a_2 \nmid q_1$. Then $x_1\equiv x_2\pmod{q/a_1}$, but $q/a_1=q_2\cdot (q_1/a_1)$, so $x_1\equiv x_2 \pmod{q_2}$; similarly, $x_1\equiv x_2\pmod{q_1}$. By the Chinese remainder theorem, $x_1\equiv x_2 \pmod{q}$.
    \item \textbf{Both $a_1$ and $a_2$ divide $q$, and $a_1=a_2$.} Here, we have $M+n_1'\equiv a_1 x\equiv M+n_2'\implies n_1'=n_2'$. Now, $M+n_1'\in \{M+1,\ldots, M+N\}$ must be a multiple of $a_1$, which occurs in at most $1+\frac{N}{a_1}$ ways. Once we choose $n_1'$, there are $a_1$ ways to choose $x$.
\end{itemize}

Putting all of this together, we have \begin{align*}
    \sum_{x=1}^{q}\nu(x)^2&\leq \sum_{a_1\neq a_2}\left(1+N\frac{(a_1,a_2)}{\max{(a_1,a_2)}}\right)+\sum_{\substack{a_1=a_2\\ a_1\nmid q}}\left(1+N\frac{(a_1,a_2)}{\max{(a_1,a_2)}}\right)+\sum_{\substack{a_1=a_2\\ a_1\mid q}}(a_1+N) \\
    &=\sum_{a_1\neq a_2}\left(1+\frac{N}{\max{(a_1,a_2)}}\right)+\sum_{a_1\nmid q}(1+N)+\sum_{a_1\mid q}(a_1+N)\\
    &=2\sum_{a_1<a_2}\left(1+\frac{N}{a_2}\right)+\sum_{a_1}(1+N)+\sum_{a_1\mid q}(a_1-1).
    \end{align*} The first two sums give \begin{align*}
    &2\sum_{2<a_2}(\pi(a_2)-1)\left(1+\frac{N}{a_2}\right)+\pi(A)(1+N)\\&= 2\sum_{2\leq i\leq \pi(A)}\left(i+\frac{Ni}{p_i}-1-\frac{N}{p_i}\right)+\pi(A)(1+N) \\
    &=2\left[\frac{\pi(A)(\pi(A)+1)}{2}-1-(\pi(A)-1)+N\sum_{2\leq i \leq \pi(A)} \frac{i-1}{p_i}\right] +\pi(A)(1+N)\\
    &=\pi(A)^2+N\pi(A)+2N\sum_{2\leq i\leq \pi(A)}\frac{i-1}{p_i}.
\end{align*} If we can show that $\sum_{2\leq i \leq \pi(A)}\frac{i-1}{p_i}\leq \li{\pi(A)}$, we will be done. By computation, we verify that this holds when $2\leq \pi(A)\leq 11$. Otherwise, we have \begin{align*}
    \sum_{2\leq i \leq \pi(A)}\frac{i-1}{p_i}&<\sum_{2\leq i \leq \pi(A)}\frac{1}{\log{i}}-\left(-\frac{1}{2}+\sum_{1\leq i \leq \pi(A)}\frac{1}{p_i}\right), \text{ by \cite[(3.12)]{RS1962},}\\
    &< \frac{1}{\log{2}}+\int_{2}^{\pi(A)}\frac{1}{\log{t}}\,dt -\log{\log{\pi(A)+\frac{1}{2}}}, \text{ by \cite[(3.19)]{RS1962},}\\
    &=\li{\pi(A)}+\frac{1}{\log{2}}-\li(2)+\frac{1}{2}-\log{\log{\pi(A)}} \\
    &\leq \li{\pi(A)}\text{, if }\pi(A)\geq 12.
\end{align*}
\end{proof}

The function $U(q,A)$ appeared in the previous lemma; we will need an upper bound for it later on, which we now give. We write $\theta(x)$ for the Chebyshev theta function.
\begin{lemma} \label{prop:U(q,A)}
    Let $q$ be a positive integer, and let $A$ be an integer greater than or equal to $\max(3,\frac{1}{36}q^{1/8}-1)$, and let $U(q,A)$ be as in Lemma \ref{prop:secondMoment}. Then $U(q,A)< 1.68\pi(A)^2$.
\end{lemma}
\begin{proof}
Fix $A$; we have $q\leq \left[36(A+1)\right]^8$. Let $k=\pi(A)$; $U(q,A)$ does not exceed the sum of the primes among $\{2,3,5,\ldots,p_k\}$ that divide $q$. If we find a $t$ with $\prod_{i=1}^{t}p_i> q$, then there are fewer than $t$ such primes, so $\frac{U(q,A)}{\pi(A)^2}< \frac{tp_k}{k^2}$. This will occur if $\prod_{i=1}^{t}p_i=\exp(\theta(p_t))> \left[36(A+1)\right]^8,$ and we claim that this occurs if $t=\ceil{3.63\log{[36(A+1)]}}$. Now, $t\geq 19$ (therefore $p_t\geq 67$), so by \cite[(3.16) and (3.12)]{RS1962} we have $$\theta(p_t)>p_t\left(1-\frac{1}{\log{p_t}}\right)>0.76217p_t>0.76217t\cdot \log{19}>8\log{[36(A+1)]},$$ as needed. Assume that $A\geq 2000$; via \cite[(3.1)]{RS1962}, we get \begin{align*}\frac{U(q,A)}{\pi(A)^2}<\frac{3.63A\log{\left[36(A+1)\right]}}{\pi(A)^2}&<\frac{3.63(\log^4{A}) \log{\left[36(A+1)\right]}}{A\left(\frac{1}{2}+\log{A}\right)^2} \\
&<3.63\frac{\log^2{A}}{\sqrt{A}}\frac{\log[36(A+1)]}{\sqrt{A}}\\
&<3.63\times 1.2919 \times 0.2502<1.174.
\end{align*} For $3\leq A< 2000$, we proceed computationally; for a given $A$, find the least $t$ such $2\cdot3\cdot\ldots\cdot p_t>q$ and pick $\hat{t}=\min(t-1,k)$. Then $$\frac{U(q,A)}{\pi(A)^2}\leq \frac{p_{k-\hat{t}+1}+p_{k-\hat{t}+2}+\ldots +p_{k}}{k^2}.$$ We calculate the right-hand side for every such $A$; every value is less than 1.68.
\end{proof}
We will encounter some expressions involving $\li(x)$ and $\pi(x)$, which we now bound.
\begin{lemma} \label{prop:piEstimate1}
    For an integer $A\geq 7$, we have $\frac{A+1}{\pi(A)}<0.45+\log{A}$.
\end{lemma}
\begin{proof}
    By \cite[(3.1)]{RS1962}, we have $\frac{A}{\log{A}}\left(1+\frac{1}{2\log{A}}\right)<\pi(A)$ for $A\geq 59$; thus $$\frac{A+1}{\pi(A)}-\log{A}<\frac{(\log{A})(\log{A}-A/2)}{(\log{A}+1/2)A}<\frac{\log{A}}{A}-\frac{1}{2},$$ which is negative for $A\geq 59$. We check the remaining values of $A$ via computation.
\end{proof}

\begin{lemma} \label{prop:piEstimate2}
    For an integer $A\geq 7$, we have $\frac{(A+1)\li{\pi(A)}}{\pi(A)^2}<2.05$.
\end{lemma}
\begin{proof}
    We have $\li(x)<\frac{x}{\log{x}}+\frac{2x}{\log^2{x}}$ for all $x>1$ (see \cite[Lemma 2.4]{Johnston2023}). We now make repeated use of \cite[(3.5)]{RS1962}, which states that $\pi(x)>\frac{x}{\log{x}}$ for $x\geq 17$. If $A\geq 17$, then\begin{align*}
        \frac{A+1}{A}\cdot \frac{A}{\pi(A)}\cdot \frac{\li{\pi(A)}}{\pi(A)}&<\frac{A+1}{A}\cdot\log{A}\cdot\left(\frac{1}{\log{A}-\log{\log{A}}}+\frac{2}{(\log{A}-\log{\log{A}})^2}\right)\\
        &=\frac{A+1}{A}\left(\frac{1}{1-\frac{\log{\log{A}}}{\log{A}}}+\frac{\frac{2}{\log{A}}}{\left(1-\frac{\log{\log{A}}}{\log{A}}\right)^2}\right) .
    \end{align*} This is decreasing when $A\geq e^e$, so for $A\geq 1000$ it is less than $1.001(1.389+0.559)<1.95.$ We check the remaining values of $A$ via computation.
\end{proof}

The following lemma is essentially trivial:
\begin{lemma} \label{prop:trivialLemma}
    Let $q\geq 1.627\times 10^{19}$ and define $$\lambda_2=2.7155\cdot (8^{\omega(q)}T(q))^{1/4}(2.578+\log{q})^{1/4}.$$ Let $\chi$ be a character of modulus $q$. If $N\leq q^{3/8}$, then $\lvert S(M,N)\rvert \leq N^{1/2} q^{3/16}\lambda_2$.
\end{lemma}
\begin{proof}
    We have $\lvert S(M,N)\rvert \leq N \leq N^{1/2}\sqrt{q^{3/8}}\leq N^{1/2}q^{3/16}\lambda_2$ because $\lambda_2\geq 1$.
\end{proof}
We now prove the main result. Our proof is quite close to that of \cite[Theorem 1.1]{JSKL2021}, so we will be brief except in those places where the proofs differ.
\begin{proof}[Proof of Theorem \ref{prop:mainTheorem}]
    If $q$ is prime, then we have $\lvert S(M,N)\rvert <3.2348 N^{1/2}q^{3/16}(\log{q})^{1/4}$ by \cite[Theorem 1.8]{Francis2021}, and we are done. Now, assume that $q$ is composite.
    If $N \leq q^{3/8}$, we are done by Lemma \ref{prop:trivialLemma} because $\lambda_2<\lambda_1$. We instead suppose that $q^{3/8}< N \leq q^{5/8}$. We let $$\mathcal{M}(y)=\max_{\substack{M,N\\N\leq y}} \lvert S(M,N)\rvert,$$ and shift by positive integers $ab$, to give $$S(M,N)=\sum_{n=M+1}^{M+N}\chi(n+ab)+2\theta \mathcal{M}(ab),$$ for some $\lvert \theta \rvert \leq 1$. We select $A=\floor{\frac{1}{36}Nq^{-1/4}}$ and $B=\floor{q^{1/4}}$; it is easily seen that $A\geq 7$ and $B\geq 1$. In \cite{JSKL2021}, they choose $A=\floor{\frac{1}{10}Nq^{-1/4}}$, but this is not optimal. Average over primes $a\leq A$ and integers $1\leq b \leq B$; then \begin{align}\label{eq:averageOverShifts}S(M,N)=\frac{1}{\pi(A)B}\sum_{n=M+1}^{M+N}\sum_{\substack{a\leq A\\a\text{ prime}}}\sum_{b=1}^{B}\chi(n+ab)+2\theta' \mathcal{M}(AB),\end{align} with $\lvert \theta' \rvert \leq 1$.
    Abbreviate the triple sum to $\sum_{n,a,b}$ and use H\"older's inequality: \begin{align*}\left\lvert \sum_{n,a,b}\chi(n+ab)\right\rvert&=\sum_{x=1}^{q} \nu(x) \left\lvert \sum_{b=1}^{B} \chi(x+b)\right\rvert\\
    \left\lvert \sum_{n,a,b} \chi(n+ab) \right\rvert^4&\leq \left(\sum_{x=1}^{q} \nu(x)\right)^{2}\left(\sum_{x=1}^{q} \nu(x)^2\right)\left(\sum_{x=1}^{q} \left\lvert \sum_{b=1}^{B}\chi(x+b)\right\rvert ^4\right).\end{align*} Now, $\sum_{x=1}^{q}\nu(x)=\pi(A)N$; the second and third sums can be bounded using Lemmas \ref{prop:secondMoment} and \ref{prop:WeilCounting} respectively (noting that $AN<q/2$ as required), to give
    \begin{align}
    \notag \left\lvert \sum_{n,a,b} \chi(n+ab) \right\rvert^4 &\leq \pi(A)^2N^2 \times [\pi(A)^2+2N\li{\pi(A)}+N\pi(A)+U(q,A)]\\
    \notag &\times \left[7B^2q+4\cdot 8^{\omega(q)}B^4T(q)\sqrt{q}\right]\\ \label{eq:averaged}
    \frac{1}{\pi(A)B}\left\lvert \sum_{n,a,b} \chi(n+ab) \right\rvert&\leq N^{1/2} q^{1/8} P(q)\left(1+\frac{2N\li{\pi(A)}}{\pi(A)^2}+\frac{N}{\pi(A)}+\frac{U(q,A)}{\pi(A)^2}\right)^{1/4},
    \end{align} where we have written $$P(q):=\left(7\left(\frac{q^{1/4}}{q^{1/4}-1}\right)^2+4\cdot 8^{\omega(q)}T(q)\right)^{1/4}\leq (7.0003+4\cdot 8^{\omega(q)}T(q))^{1/4}.$$ Using Lemmas \ref{prop:U(q,A)}, \ref{prop:piEstimate1} and \ref{prop:piEstimate2}, the bracketed term in \eqref{eq:averaged} is not more than \begin{align*}&q^{1/4}\left(2.68q^{-1/4}+\frac{72(A+1)\li{\pi(A)}}{\pi(A)^2}+\frac{36(A+1)}{\pi(A)}\right)\\
    &\leq q^{1/4}(0.0000422+72(2.05)+36(0.45+\log{A}))\\
    &\leq q^{1/4}\left(163.8000422+36\left(\frac{3}{8}\log{q}-\log{36}\right)\right)\leq q^{1/4}\left(34.8+\frac{27}{2}\log{q}\right)
    \end{align*}
    Putting this back into \eqref{eq:averaged}, we get \begin{align*}
        \frac{1}{\pi(A)B}\left\lvert \sum_{n,a,b} \chi(n+ab) \right\rvert &\leq \left(\frac{27}{2}\right)^{1/4}N^{1/2}q^{3/16}\left(8^{\omega(q)}T(q)\right)^{1/4}\left(\frac{7.0003}{256}+4\right)^{1/4}(2.578+\log{q})^{1/4}\\
        &\leq 2.7155\cdot N^{1/2}q^{3/16}(8^{\omega(q)}T(q))^{1/4}(2.578+\log{q})^{1/4},
    \end{align*} where we have used the fact that $q$ is composite.
    With $\lambda_2$ as in Lemma \ref{prop:trivialLemma}, \eqref{eq:averageOverShifts} gives $\mathcal{M}(N)\leq N^{1/2}q^{3/16}\lambda_2+2\mathcal{M}(N/36)$ for any positive $N\leq q^{5/8}$. Repeated application of this result yields $$\mathcal{M}(N)\leq N^{1/2}q^{3/16} \lambda_2 \sum_{k=0}^{K}\left(\frac{2}{\sqrt{36}}\right)^k+2^{K+1}\mathcal{M}\left(\frac{N}{36^{K+1}}\right),$$ for any integer $K\geq 0$. Bounding the second term trivially and taking $K\rightarrow \infty$ gives $$\mathcal{M}(N)\leq N^{1/2}q^{3/16}\lambda_2 \left(\frac{1}{1-\frac{2}{\sqrt{36}}}\right)\leq N^{1/2}q^{3/16}\lambda_1.$$

\end{proof}
\section*{Acknowledgements}
The author thanks his supervisor Tim Trudgian for suggesting this project and for his assistance throughout. The author also acknowledges the support of an Australian Government Research Training Program Scholarship.

\bibliographystyle{plain}
\bibliography{refs}
\end{document}